\documentclass[10pt, reqno]{amsart}
\usepackage{lipsum}
\usepackage{a4wide}
\usepackage{amssymb} 
\usepackage{amsmath}
\usepackage{amsthm} 
\usepackage{tikz} 
\usepackage{amscd}
\usepackage{hyperref}
\usepackage{enumitem}
\hypersetup{colorlinks,linkcolor={blue},citecolor={blue},urlcolor={red}}
\theoremstyle{plain}
\numberwithin{equation}{section}

\newtheorem{theorem}{Theorem}[section]
\newtheorem{corollary}[theorem]{Corollary} 

\newtheorem{remark}[theorem]{Remark}
\newtheorem{proposition}[theorem]{Proposition}
\newtheorem{conjecture}[theorem]{Conjecture}

\title{A note on finiteness of Tate cohomology groups}
\author{Sabyasachi Dhar and Santosh Nadimpalli}
\date{\today}
\begin{document}
\begin{abstract}
    Let $G$ be a reductive algebraic group defined over a non-Archimedean local
    field $F$ of residue characteristic $p$. Let 
    $\sigma$ be an automorphism of $G$ of order $\ell$--a prime number--with
    $\ell\neq p$. Let $\Pi$ be a finite length 
    $\overline{\mathbb{F}}_\ell$-representation of 
    $G(F)\rtimes \langle\sigma\rangle$. We show that the Tate 
    cohomology $\widehat{H}^i(\langle\sigma\rangle, \Pi)$
    is a finite length representation of  $G^\sigma(F)$. We give an
    application to genericity of these Tate cohomology spaces. 
\end{abstract}
\maketitle
\section{Introduction}
Let $G$ be a reductive algebraic group defined over a non-Archimedean
local field $F$ of residue characteristic $p$. Let $\sigma$ be an
automorphism of $G$, defined over $F$, and we assume that the order of
$\sigma$ is a prime number $\ell\neq p$. We denote by $\Gamma$ the
group generated by $\sigma$. For a smooth $\overline{\mathbb{F}}_\ell$
(or $\ell$-modular) representation $(\Pi, V)$ of $G(F)\rtimes \Gamma$,
the Tate cohomology groups for the action of the cyclic group
$\Gamma$, denoted by $\widehat{H}^i(\Gamma, \Pi)$, are representations
of $G^\sigma(F)$, where $G^\sigma$ is the group of $\sigma$-fixed
points of $G$.  In the work of Treumann and Venkatesh \cite{TV_paper}
on the mod-$\ell$ functoriality lifting of systems of Hecke
eigenvalues from $H^\sigma$ to $H$ (here, $H$ is a reductive algebraic
group defined over a number field $K$ and $\sigma$ is an automorphism
of $H$ defined over $K$), they conjectured that, for any finite length
$\ell$-modular representation $(\Pi, V)$ of $G(F)\rtimes \Gamma$, the
$G^\sigma(F)$ representations $\widehat{H}^i(\Gamma, \Pi)$ are of
finite length, for $i\in \{0,1\}$ (see \cite[Conjecture
6.3]{TV_paper}). Several significant cases are proved in our previous
works (see \cite[Proposition 5.5]{Nadimpalli_2024_GLn} and
\cite[Theorem 3.4.1]{dhar2025jacquetmodulestatecohomology}) and in
this note, we prove this finiteness result in complete generality. We
prove:
\begin{theorem}\label{main}
  Let $G$ be a reductive algebraic group defined over a
  non-Archimedean local field $F$ of residue characteristic $p$, and
  let $\sigma$ be an automorphism of $G$, defined over $F$, and has
  prime order $\ell\neq p$. Set $\Gamma=\langle\sigma\rangle$, and let
  $\Pi$ be a finite length $\ell$-modular representation of
  $G(F)\rtimes \Gamma$.  Then the Tate cohomology spaces
  $\widehat{H}^i(\Gamma, \Pi)$ are finite length $\ell$-modular
  representations of $G^\sigma(F)$.
\end{theorem}
\section{proof of Theorem \ref{main}}
In our article \cite[Theorem
2.3]{dhar2025jacquetmodulestatecohomology}, we showed that the Tate
cohomology spaces $\widehat{H}^i(\Gamma, \Pi)$, for $i\in\{0,1\}$, are
finitely generated representations of $G^\sigma(F)$, and it remains to
show that these cohomology groups are admissible representations of
$G^\sigma(F)$.  Note that the main technical inputs for the finite
generation of Tate cohomology groups are Schneider and Stuhler
resolutions.  The hurdle to prove admissibility is with the case where
the rank of a maximal $F$-split torus in $Z(G^\sigma)$ can be larger
than that of the rank of a maximal $F$-split torus of $Z(G)$. The
following proposition (Proposition \ref{prop}) reduces Theorem
\ref{main} to the case where the maximal $F$-split torus of
$Z(G^\sigma)$ is contained in $Z(G)$.
 
Let $C$ be the maximal $F$-split torus of $Z(G^\sigma)$. Let $S$ be a
maximal torus of $(G^\sigma)^0$ defined over $F$, where $(G^\sigma)^0$
is the connected component of the identity element of $G^\sigma$.  The
centralizer $Z_G(S)$ is a $\sigma$-stable maximal torus of $G$ defined
over $F$, and we denote it by $T$ (see \cite[Proposition
12.8.1(2)]{kaletha_prasad}).  Let $\Phi(G, T)$ be the set of roots of
$T$ in $G$. Note that $\Phi(G, T)$ is the absolute root system of $G$.
Let $M$ be the centralizer of $C$ in $G$ (denoted by $Z_G(C)$).  The
group $M$ is a $\sigma$-stable Levi subgroup of an $F$-rational
parabolic subgroup $P$ of $G$ (see \cite[Proposition 20.4]{borel_book}
and \cite[Theorem 4.15]{borel_tits_gr}). We want to claim that $M$ is
a Levi-subgroup of a $\sigma$-stable parabolic subgroup $P$ of $G$.
For this, we follow the arguments in the proof of \cite[Proposition
20.4]{borel_book}. Let $\lambda\in X_\ast(C)$ be a non-zero element
such that for all $\alpha\in \Phi(G, T)$ which restrict to a non-zero
element in $X^\ast(C)$, we have $\langle\alpha, \lambda\rangle\neq
0$. We set
$$\Psi=\{\alpha\in \Phi(G, T): \langle\alpha, \lambda\rangle>0\}.$$
Note that 
$$\Phi(G, T)=\Phi(Z_G(C), T)\coprod \Psi\coprod -\Psi.$$
The unipotent group $U_\Psi$ generated by the root groups $U_\beta$,
for $\beta\in \Psi$, is a $\sigma$-stable subgroup defined over
$F$. The group $P=Z_G(C)U_\Psi$ is a $\sigma$-stable parabolic
subgroup defined over $F$ and contains $M=Z_G(C)$ as its Levi
subgroup.  The unipotent radical of $P$, i.e., the group $U_\Psi$,
will be denoted by $U$.  The group $U$ is $\sigma$-stable.  By
definition, $G^\sigma$ is contained in $M$, and thus, we get that
$M^\sigma=G^\sigma$. Observe that the maximal $F$-split torus in
$Z(M^\sigma)=Z(G^\sigma)$ is contained in $Z(M)$.  We denote by
$\Pi_{U(F)}$, the Jacquet module of $\Pi$ with respect to $U(F)$, and
let $\Pi(U(F))$ be the kernel of the natural map
$\Pi\rightarrow \Pi_{U(F)}$. Note that the space $\Pi(U(F))$ is
$\sigma$-stable.
\begin{proposition}\label{prop}
  Let $\Pi$ be a smooth $\ell$-modular representation of $G(F)$ such
  that $\Pi$ extends as a representation of $G(F)\rtimes \Gamma$. Then
  the Tate cohomology space $\widehat{H}^i(\Gamma, \Pi(U(F)))$ is
  trivial, for $i\in \{0,1\}$.
\end{proposition}
\begin{proof}
  The above proposition can be derived from Glauberman's
  correspondence applied to the $p$-group $U_0/U_1$, where $U_0$ and
  $U_1$ are $\sigma$-stable compact open subgroups of $U(F)$ such that
  $U_1$ is a normal subgroup of $U_0$. Since $U(F)$ does not have
  $\sigma$-fixed points, and from the fact that $U_0$ is a pro-$p$
  group, we get that $(U_0/U_1)^\sigma$ is a trivial group.  If
  $(\rho, V)$ is an $\ell$-modular representation of
  $(U_0/U_1)\rtimes \Gamma$, then we see that the Tate cohomology
  $\widehat{H}^i(\Gamma, \rho)$ decomposes into Tate cohomology of
  $\sigma$-stable sub-representations of $\rho$. But the only
  $\sigma$-stable sub-representations of $U_0/U_1$ are trivial
  representations by Glauberman's correspondence. We can apply these
  arguments to $\Pi(U(F))$ for $\sigma$-stable representations spanned
  by vectors which are annihilated by the augmentation ideal of some
  $U_0/U_1$ as above. We choose to give a direct argument below
  without appealing to Glauberman's correspondence.

Let us first consider the case $i=0$.
Let $w\in \Pi(U(F))^\sigma$ and let $U_0$ be a $\sigma$-stable
 compact open subgroup of $U(F)$
 such that 
 \begin{equation}\label{argu}
\int_{U_0}\Pi(u)w\,du=0,
 \end{equation}
 where $du$ is a Haar measure on $U_0$. Let $U_1$ be a $\sigma$-stable
 open normal subgroup of $U_0$ such that $w\in \Pi^{U_1}$. Let $H$ be
 the quotient group $U_0/U_1$. Then we have
$$ \int_{U_0}\Pi(u)w\,du = \sum_{h\in H} \Pi(h)w = \sum_{h\in H^\sigma} \Pi(h)w +
\sum_{h\in H\setminus H^\sigma} \Pi(h)w. $$ Since the $\Gamma$-action
on $H\setminus H^\sigma$ is free, there exists a subset
$X\subset H\setminus H^\sigma$ such that
$H\setminus H^\sigma = \coprod_{i=0}^{\ell-1}\sigma^iX$, and we get
the following equality
$$
\sum_{h\in H} \Pi(h)w = \sum_{h\in H^\sigma} \Pi(h)w + {\rm
  Nr}\big(\sum_{x\in X}\Pi(x)w\big).
$$
Note that $H^\sigma = \{e\}$ as $U_i$ are $p$-groups and $\sigma$ is
an element of order $\ell$. Then it follows from the above equation
that
$$ \sum_{h\in H} \Pi(h)w = w $$
in $\widehat{H}^0(\Gamma, \Pi(U(F)))$. Now, using the equation
(\ref{argu}), we get that $w=0$ in $\widehat{H}^0(\Gamma,\Pi(U(F)))$.

We now consider the case where $i=1$. 
Let $v\in {\rm Ker}\big({\rm Nr}:\Pi(U(F))\rightarrow \Pi(U(F))\big)$. There 
exists a $\sigma$-stable compact open subgroup $N_0$ of $U(F)$ such that
$$ \int_{N_0} \Pi(n)v\,dn = 0, $$
where $dn$ is a Haar measure on $N_0$. Let $N_1$ be a $\sigma$-stable
open normal subgroup of $N_0$ such that $v\in \Pi^{N_1}$. If
$\mathcal{N}$ denotes the quotient group $N_0/N_1$, then we have
$$ \int_{N_0} \Pi(n)v\,dn = \sum_{n\in \mathcal{N}} \Pi(n)v 
= \sum_{n\in \mathcal{N}^\sigma}\Pi(n)v + \sum_{n\in \mathcal{N}\setminus 
\mathcal{N}^\sigma}\Pi(n)v. $$
Note that $\mathcal{N}^\sigma = \{e\}$ as $N_i$ are $p$-groups and $\sigma$ 
is an element of order $\ell$, and there exists a subset $X\subseteq 
\mathcal{N}\setminus \mathcal{N}^\sigma$ such that $\mathcal{N}\setminus 
\mathcal{N}^\sigma = \coprod_{i=0}^{\ell-1}\sigma^iX$. Then the above 
identity becomes
\begin{equation}\label{argu_3}
v \,\,+ \,\,\sum_{i=0}^{\ell-1}\sum_{x\in X}\Pi(\sigma^i(x))v = 0.
\end{equation}
Set $w = \sum_{i=0}^{\ell-1}\sum_{x\in X}\Pi(\sigma^i(x))v$. We will show 
that $w\in {\rm Im}(1-\sigma)$. Using the fact that ${\rm Nr}(v) = 0$, we have
\begin{align*}
w 
&= \sum_{x\in X}\Pi(x)v \,\,+\,\, \sum_{i=1}^{\ell-1}\sum_{x\in X} \Pi(\sigma^i(x))v\\
&= -\sum_{i=1}^{\ell-1}\sum_{x\in X}
\Pi(x)\sigma^{i}v \,\,+\,\, \sum_{i=1}^{\ell-1}\sum_{x\in X} \Pi(\sigma^i(x)) v\\
&= -\sum_{i=1}^{\ell-1}\sigma^i\sum_{x\in X} \Pi(\sigma^{\ell-i}(x))v 
\,\,+\,\, \sum_{i=1}^{\ell-1}\sum_{x\in X}
\Pi(\sigma^{\ell-i}(x)) v\\
&= \sum_{i=1}^{\ell-1}(1-\sigma^i)\sum_{x\in X}\Pi^{\sigma^{\ell-i}}(x)v.
\end{align*}
Thus, we get that $w = (1-\sigma)w_1$ for some $w_1\in \Pi(U(F))$, and
hence it follows from (\ref{argu_3}) that $v = 0$ in
$\widehat{H}^1(\Gamma, \Pi(U(F))$.
\end{proof}
\begin{corollary}
  Let $G$ be a reductive algebraic group and let $P$ be a
  $\sigma$-stable parabolic subgroup with unipotent radical $U$ and a
  $\sigma$-stable Levi subgroup $M$. Assume that $U\cap G^\sigma$ is
  trivial and $G^\sigma\subset M$.  Then, the natural map
     $$\widehat{H}^i(\Gamma, \Pi)\rightarrow \widehat{H}^i(\Gamma, \Pi_{U(F)})$$
     is an isomorphism. 
\end{corollary}
\begin{corollary}
  Let $G$ be a reductive algebraic group defined over a $p$-adic field
  $F$ and let $s$ be an element of order $\ell\neq p$ such that the
  group $C_G(s)$ is contained in a proper parabolic subgroup of $G$.
  Let $\Pi$ be an $\ell$-modular cuspidal representation of
  $G(F)$. Then the Tate cohomology space
  $\widehat{H}^i(\langle s\rangle, \Pi)$ is trivial, for
  $i\in \{0,1\}$.
\end{corollary}
We resume the proof of Theorem \ref{main}. Here, we will use the fact
which was also used in our previous work \cite[Theorem 2.3]
{dhar2025jacquetmodulestatecohomology}, that the sub-representations
of finitely generated smooth $\ell$-modular representations of $G(F)$
are again finitely generated. This is a fundamental result that
follows from the works of \cite[Corollary 1.3]{DHMK_finiteness} and
\cite[Corollary 4.5] {Dat_finiteness_p_adic}.  From now, we assume
that the maximal $F$-split torus $C$ of $Z(G^\sigma)$ is contained in
$Z(G)$.  For a smooth representation $\pi$ of $G^\sigma(F)$, let
$\Phi$ be the following map obtained by Frobenius reciprocity map
 $$\Phi:\pi\rightarrow \prod_{P=MN, P\neq 
   G^\sigma}i_{P(F)}^{G^\sigma(F)}\pi_{N(F)}$$ and let $\pi_c$ be the
 kernel of $\Phi$. Let $P$ be a proper parabolic subgroup of
 $G^\sigma$ with unipotent radical $U$. We note that the Jacquet
 modules $({\pi_c})_{U(F)}$ is trivial. Applying $\Phi$ to the Tate
 cohomology space $\widehat{H}^0(\Gamma, \Pi)$, we get the following
 exact sequence
 \begin{equation}\label{conclude_exact_sequence}
   0\rightarrow  \widehat{H}^i(\Gamma, \Pi)_c\rightarrow
   \widehat{H}^i(\Gamma, \Pi)\xrightarrow{\Phi}
   \prod_{P=MN, P\neq G^\sigma}i_{P(F)}^{G^\sigma(F)}[\widehat{H}^i(\Gamma, \Pi)_{N(F)}]
\end{equation}
Note that $\widehat{H}^i(\Gamma, \Pi)_c$ is a cuspidal representation,
and it is a $G^\sigma(F)$ sub-representation of the finitely generated
representation $\widehat{H}^i(\Gamma, \Pi)$.  This implies that
$\widehat{H}^i(\Gamma, \Pi)_c$ is an $\ell$-modular finitely generated
cuspidal representation of $G^\sigma(F)$, and using the fact that the
split component $C$ of $Z(G^\sigma)$ is contained in $Z(G)$, we get
that $\widehat{H}^i(\Gamma, \Pi)_c$ is a finite length representation
of $G^\sigma(F)$.  Let $Q$ be a $\sigma$-stable parabolic subgroup
such that $Q^\sigma=P$, and let $U$ be the unipotent radical of
$Q$. Let $L$ be a $\sigma$-stable Levi subgroup of $Q$ such that the
connected component of $L^\sigma$ is a Levi subgroup $M'$ of
$P$. Then, using \cite[Theorem
3.3.2]{dhar2025jacquetmodulestatecohomology} we have
$$ \widehat{H}^i(\Gamma, \Pi)_{N(F)}\subset 
\widehat{H}^i(\Gamma, \Pi_{U(F)}). $$ Then, using induction, we get
that the Tate cohomology space $\widehat{H}^i(\Gamma, \Pi_{U(F)})$ is
a finite length representation of $M'(F)$. Since parabolic induction
preserves finite length, it follows from exact sequence
(\ref{conclude_exact_sequence}) that $\widehat{H}^i(\Gamma, \Pi)$ is a
finite length representation of $G^\sigma(F)$.
\begin{section}{Genericity of Tate cohomology}
  In this section, we give an application of Theorem \ref{main} to the
  genericity of Tate cohomology spaces.  Let $G$ be a quasi-split
  reductive algebraic group defined over a $p$-adic field $F$, and let
  $\sigma\in {\rm Aut}(G)(F)$ be an element of prime order $\ell$. We
  denote by $\Gamma$ the group generated by $\sigma$. Let
  $\mathfrak{g}$ be the Lie algebra of $G$, and the $\sigma$-fixed
  point Lie subalgebra $\mathfrak{g}^\sigma$ is the Lie algebra of
  $G^\sigma$.  Let $B_\mathfrak{g}$ be the Killing form of $G$.  We
  assume that $G^\sigma$ is a quasi-split reductive algebraic group
  and that there exists an element $X\in \mathfrak{g}^\sigma$ such
  that $X$ is regular nilpotent in $\mathfrak{g}$. This implies that
  $X$ is regular in $\mathfrak{g}^\sigma$.  Let
  $\varphi:\mathbb{G}_m\rightarrow G$ be a co-character such that
  $X\in \mathfrak{g}_{-2}$, where
  $\mathfrak{g}=\bigoplus_{i\in \mathbb{Z}}\mathfrak{g}_i$ is the
  grading induced by $\varphi$. Let $N$ be the closed subgroup of $G$
  with Lie algebra $\bigoplus_{i>0}\mathfrak{g}_i$, and let $B$ be the
  Borel subgroup normalising $N$.  Let $B^-$ be the Borel subgroup of
  $G$ with Lie algebra $\bigoplus_{i\leq 0}\mathfrak{g}_i$. Note that
  the group $B$ is stable under $\Gamma$, and $(B^\sigma)^0$ is a
  Borel subgroup of $(G^\sigma)^0$. Note that $N$ is the unipotent
  radical of $B$ and $N^\sigma$ is the unipotent radical of
  $(B^\sigma)^0$.

  Let $\mathfrak{o}_F$ be the ring of integers of $F$, and let
  $\mathfrak{p}_F$ be the maximal ideal of $F$. Fix an uniformizer
  $\varpi$ of $F$. Let
  $\psi:F\rightarrow \overline{\mathbb{Z}}_\ell^\times$ be an additive
  character which is trivial on $\mathfrak{o}_F$ and non-trivial on
  $\mathfrak{p}_F$. We denote by
  $\overline{\psi}:F\rightarrow \overline{\mathbb{F}}_\ell^\times$ the
  mod-$\ell$ reduction of $\psi$.  Let
  $\psi_X:N(F)\rightarrow \overline{\mathbb{Z}}_\ell^\times$ be the
  character
$$\psi_X({\rm exp}(Y))=\psi({\rm B}_\mathfrak{g}(X, Y)).$$
The mod-$\ell$ reduction of $\psi_X$ is denoted by
$\overline{\psi}_X$.  Recall that for an irreducible smooth
$\ell$-modular representation $(\Pi,V)$ of $G(F)$, the multiplicity
one statement
$$ {\rm dim}_{\overline{\mathbb{F}}_\ell}(V_{N(F),\overline{\psi}_X}) \leq 1.$$
follows from \cite[Corollary 7.2]{nadir_whit_model}, and the representation
$(\Pi,V)$ is said to be $(N(F),\overline{\psi}_X)$-generic if 
${\rm dim}_{\overline{\mathbb{F}}_\ell}(V_{N(F),\overline{\psi}_X})=1$. 
Note that $\Gamma$ acts trivially on $V_{N(F),\overline{\psi}_X}$. 
We now prove the following theorem.
\begin{theorem}
  Let $(\Pi, V)$ be an irreducible smooth $\ell$-modular
  representation of $G(F)$ such that $\Pi$ extends as a representation
  of $G(F)\rtimes\Gamma$.  Let
  $\overline{\psi}_X:N(F)\rightarrow
  \overline{\mathbb{F}}_\ell^\times$ be a (necessarily non-degenerate)
  character of $N(F)$ associated with a $\sigma$-invariant element
  $X\in \mathfrak{g}$ such that $\overline{\psi}_X$ restricts to a
  non-degenerate character of $N^\sigma(F)$.  If $(\Pi, V)$ is
  $(N(F), \overline{\psi}_X)$-generic, then
  $\widehat{H}^i(\Gamma, \Pi)$ has a unique
  $(N^\sigma(F), \overline{\psi}_X)$-generic subquotient, for
  $i\in \{0,1\}$.
\end{theorem}
\begin{proof}
  We will use the fact that the natural map
  $V\rightarrow V_{N(F), \overline{\psi}_X}$ is an isomorphism on
  semi-invariants with respect to a compact open subgroup $K$ of
  $G(F)$, this important result is due to Rodier in the case of smooth
  complex representations, and it is later generalised by Moeglin and
  Waldspurger (\cite{MW_paper}) for degenerate Whittaker models. In
  \cite[Section III]{Vigneras_Highest_Whittaker_model}, Vign\'eras
  used these ideas to prove $G(F)$-stability of certain natural
  lattices in the Whittaker model of a smooth irreducible
  representation of $G(F)$ on a $\overline{\mathbb{Q}}_\ell$-vector
  space. We will recall the construction of the group $K$ following
  Moeglin--Waldspurger and Vigneras's paper. We need to verify that
  $K$ is $\sigma$-stable.

Let $L$ be a $\sigma$-stable lattice in $\mathfrak{g}$ such 
that $[L, L]\subset L$. Let $\mathfrak{g}^X$ be the 
centralizer in $\mathfrak{g}$ of the element $X$. Let $\mathfrak{m}$ be a 
$\varphi$-stable complement of $\mathfrak{g}^X$. Note that
we can choose $\mathfrak{m}$ such that 
$\sigma(\mathfrak{m})=\mathfrak{m}$. 
Note that the nil-radical of the alternating form 
$B_X(Y, Z)=B_\mathfrak{g}(X, [Y\ Z])$
is $\mathfrak{g}^X$. Choose a $\sigma$-stable lattice 
$M_1$ in $\mathfrak{m}$ such that 
$B_X(M_1, M_1)\subset \mathfrak{o}_F$. 
The lattice 
$$L_1=M_1\oplus \sum_{i\in \mathbb{Z}}
(L\cap \mathfrak{g}^X\cap \mathfrak{g}_i)$$ defines the groups
$G_n={\rm exp}(\varpi^n L_1)$, for all $n\geq B$, where $B$ is a
constant defined in \cite[Lemma I.3]{MW_paper}. Let $\chi_n$ be the
character of $G_n$ given by
$${\rm exp}(Z)\mapsto \psi(B_\mathfrak{g}(\varpi^{-2n}X, Z)),$$
for all $Z\in \varpi^nL_1$. 
Let $t=\varphi(\varpi)$
and let 
$$K_n=t^{-n}G_nt^n,\ \chi_n'(k)=\chi_n(t^nkt^{-n}).$$
Note that $\chi_n'(x)=\psi_{X}(x)$, for $x\in N(F)\cap K_n$, and 
$\chi_n'(x)=1$, for $x\in B^-(F)\cap K_n$. 
The mod-$\ell$ reduction of $\chi_n'$ is denoted by 
$\overline{\chi}_n'$.
Note that $K_n$ and $\overline{\chi}'_n$ are $\Gamma$-invariant. As in the 
complex case, we claim 
that the natural map 
\begin{equation}\label{section}
    V^{K_n, \overline{\chi}_n'}\rightarrow V_{N(F), \overline{\psi}_X}
\end{equation}
is an isomorphism for all $n$ large enough. To see this, given a 
$(N(F), \overline{\psi}_X)$-generic $\ell$-modular representation 
$(\Pi, V)$, there exists an $\ell$-adic integral 
representation $\widetilde{\Pi}$ of $G(F)$ such that $\Pi$ is the 
unique $(N(F), \psi_X)$-generic component of the mod-$\ell$
reduction of $\widetilde{\Pi}$ 
(see \cite[Corollary 7.2]{nadir_whit_model}). In other words, there is a 
$G(F)$-stable lattice 
$\mathcal{L}$ in the $\overline{\mathbb{Q}}_\ell$-representation
$\widetilde{\Pi}$ such that 
$\mathcal{L}\otimes_{\overline{\mathbb{Z}}_\ell}\overline{\mathbb{F}}_\ell$
contains $\Pi$ as the unique $(N(F), \overline{\psi}_X)$-generic component.  
 Note that Vign\'eras (\cite[Proposition III.4.6]
{Vigneras_Highest_Whittaker_model})
showed that 
$$\mathcal{L}^{K_n, \chi_n'}\rightarrow \mathcal{L}_{N(F), \psi_X}$$
is an isomorphism, and clearly both sides are rank one free modules over 
$\overline{\mathbb{Z}}_\ell$. 
Since $K_n$ is a pro-$p$ group, we have 
$$V^{K_n, \overline{\chi}_n'}\subset \mathcal{L}^{K_n, 
\chi_n'}\otimes_{\overline{\mathbb{Z}}_\ell}
\overline{\mathbb{F}}_\ell,$$
and 
$$\mathcal{L}_{N(F),\psi_X}
\otimes_{\overline{\mathbb{Z}}_\ell}
\overline{\mathbb{F}}_\ell
=V_{N(F), \overline{\psi}_X}.$$
It remains to show that $V^{K_n, \chi_n'}$ is a non-zero
space for $n$ large, and to show this we follow 
the arguments in \cite[Lemme I.10]{MW_paper}. 
Consider the quotient  map 
$$V\rightarrow V_{N(F), \overline{\psi}_X}$$
and let $v\in V$ be an element which maps to a non-zero 
element in $V_{N(F), \overline{\psi}_X}$. 
There exists a large $n$
such that $v\in V^{K_n\cap B^-(F)}$. Note that 
the element 
$$I_n'(v)=\int_{N(F)\cap K_n} (\overline{\chi}_n')^{-1}(x)\Pi(x)v\,dx $$
has the same image as $v$ in $V_{N(F), \overline{\psi}_X}$ and
we have $I_n'(v)\in V^{K_n,\overline{\chi}_n'}\backslash \{0\}$.
This shows that the natural map \eqref{section} is an isomorphism.

When $(\Pi,V)$ extends as a
representation of $G(F)\rtimes \Gamma$, we observe that the map 
\eqref{section} is a $\sigma$-equivariant map. Note that the 
group $K_n$ is a pro-$p$ group. Thus, we may find a $\sigma$-equivariant
splitting 
$$V=W\oplus V^{K_n, \overline{\chi}_n'},$$
and we get that the natural map 
$$\widehat{H}^i(\Gamma, V)\rightarrow 
\widehat{H}^i(\Gamma, V^{K_n, \overline{\chi}_n'})\xrightarrow{\simeq}
\widehat{H}^i(\Gamma, V_{N(F), \overline{\psi}_X})$$ is non-zero for
$i\in \{0,1\}$.  It follows from \cite[Theorem
3.3.2]{dhar2025jacquetmodulestatecohomology} that the natural map
$$\widehat{H}^i(\Gamma, V)\rightarrow \widehat{H}^i
(\Gamma, V_{N(F), \overline{\psi}_X})$$
factorises through an injective map 
$$\widehat{H}^i(\Gamma, V)_{N^\sigma(F), \overline{\psi}_X}
\rightarrow \widehat{H}^i(\Gamma, V_{N(F), \overline{\psi}_X}),$$
and it is non-zero as noted above. Hence,
the above map is an isomorphism. 
Thus, from our result on 
finiteness of Tate cohomology (Theorem \ref{main}), 
and the uniqueness of Whittaker model in the 
$\ell$-modular case 
\cite[Corollary 7.2]{nadir_whit_model}, we get that 
there exists a unique generic sub-quotient of 
$$\widehat{H}^i(\Gamma, \Pi).$$
This proves the theorem. 
\end{proof}
\begin{remark}\normalfont
The above theorem applies to the case where $G={\rm Res}_{E/F}H$,
where $H$ is a quasi-split reductive algebraic group defined over $F$, 
and $E/F$ is a cyclic Galois extension of degree $\ell$. The 
automorphism $\sigma$ is induced by a generator of ${\rm Gal}(E/F)$, and 
$G(F)=H(E)$ and $G^\sigma(F)=H(F)$. The theorem also covers the
case where $G={\rm GL}_{2n+1}$
and $G^\sigma={\rm SO}_{2n+1}$, and the case where $G={\rm GL}_{2n}$ and 
$G={\rm Sp}_{2n}$ these are some examples from classical groups.
Note that the case where $\sigma$ is a triality automorphism on the 
group $G$ of the type ${\rm D}_4$ is also an example where the theorem is 
applicable. The above theorem does not deal with the case where 
$G={\rm GL}_{2n}$ and $G^\sigma={\rm SO}_{2n}$.
\end{remark}

\end{section}
\bibliography{./biblio} 
\bibliographystyle{amsalpha}
\vspace{0.2 cm}
Sabyasachi Dhar, \\
Department of Mathematics and Statistics, Indian
Institute of Technology Bombay, Maharashtra 400076, India.\\
\texttt{mathsabya93@gmail.com},
\texttt{sabya@math.iitb.ac.in}.
\vspace{0.1 cm}\\
Santosh Nadimpalli,\\
Department of Mathematics and Statistics, Indian
Institute of Technology Kanpur, U.P. 208016, India.\\
\texttt{nvrnsantosh@gmail.com}, \texttt{nsantosh@iitk.ac.in}.

\end{document}